\documentclass[
a4paper,12pt]{amsart}
\usepackage{version}
\usepackage{amssymb}
\usepackage{amsfonts}
\usepackage{graphicx}
\usepackage{color}
\usepackage{epstopdf}
\usepackage[english]{babel}
\usepackage{float}
\usepackage{cite}
\usepackage{tikz,pgf,pgfplots,wrapfig}
\pgfplotsset{compat=1.15, ticks=none}
\usetikzlibrary{arrows}
\usetikzlibrary{patterns,spy}

\usepackage{geometry}
\newtheorem{theorem}{Theorem}[section]

\newtheorem{lemma}[theorem]{Lemma}
\newtheorem{corollary}[theorem]{Corollary}
\newtheorem{proposition}[theorem]{Proposition}

\theoremstyle{definition}
\newtheorem{definition}[theorem]{Definition} 
\newtheorem{example}[theorem]{Example}

\theoremstyle{remark}
\newtheorem{remark}[theorem]{Remark}

\numberwithin{equation}{section}

\renewcommand\bigskip{\medskip}

\def\cF{\mathcal F}

\def\N{\mathbb N}

\def\R{\mathbb R}

\def\Z{\mathbb Z}

\def\vep{\varepsilon}

\def\rmin{r_{\min}}

\DeclareMathOperator{\dimH}{dim_H}

\DeclareMathOperator{\dimA}{dim_A}
\DeclareMathOperator{\dimB}{dim_B}
\DeclareMathOperator{\udimB}{\overline{dim}_B}
\DeclareMathOperator{\dist}{dist}
\DeclareMathOperator{\diam}{diam}

\DeclareMathOperator{\Pen}{Pen}
\DeclareMathOperator{\Int}{Int}
\DeclareMathOperator{\Id}{Id}
\DeclareMathOperator{\graph}{graph}
\DeclareMathOperator{\dH}{d_H}
\DeclareMathOperator{\Conv}{Conv}

\newcommand{\bi}{{\bf i}}
\newcommand{\bj}{{\bf j}}

\begin{document}

\title[Visible parts of self-similar sets]{On dimensions of visible parts of
  self-similar sets with finite rotation groups}

\author[Esa J\"arvenp\"a\"a]{Esa J\"arvenp\"a\"a}
\address{Department of Mathematical Sciences, P.O. Box 3000, 90014 
University of Oulu, Finland}
\email{esa.jarvenpaa@oulu.fi}

\author[Maarit J\"arvenp\"a\"a]{Maarit J\"arvenp\"a\"a}
\address{Department of Mathematical Sciences, P.O. Box 3000, 90014 
University of Oulu, Finland}
\email{maarit.jarvenpaa@oulu.fi}

\author[Ville Suomala]{Ville Suomala}
\address{Department of Mathematical Sciences, P.O. Box 3000, 90014 
University of Oulu, Finland}
\email{ville.suomala@oulu.fi}

\author[Meng Wu]{Meng Wu}
\address{Department of Mathematical Sciences, P.O. Box 3000, 90014 
University of Oulu, Finland}
\email{meng.wu@oulu.fi}

\thanks{The fourth author is supported by the Academy of Finland, project grant
No. 318217. We acknowledge the support of the Centre of Excellence in Analysis
and Dynamics Research funded by the Academy of Finland.}

\subjclass[2010]{primary 28A80; secondary 28D05, 37A05.}
\keywords{Visible part, dimension, self-similar set}

\begin{abstract}
We derive an upper bound for the Assouad dimension of visible parts of
self-similar sets generated by iterated function systems with finite rotation
groups and satisfying the weak separation condition. The bound is valid for all
visible parts and it depends on the direction and the penetrable part of the set, which is a
concept defined in this paper. As a corollary, we obtain in the planar case that
if the projection is a finite or countable union of intervals then the visible
part is 1-dimensional. We also prove that the Assouad dimension of a visible
part is strictly smaller than the Hausdorff dimension of the set provided the
projection contains interior points.  Our proof relies on Furstenberg's dimension conservation principle for self-similar sets. 
\end{abstract}

\maketitle


\section{Introduction}\label{intro}

The visible part of a compact set $K\subset\R^d$ from a hyperplane $L$ consists
of those points of $K$ which one can see from $L$ (for a precise definition, see
\eqref{visibledef}). This definition extends naturally to $k$-dimensional planes
$L$ for $k=0,\dots,d-1$ but, in this paper, we concentrate on the case $k=d-1$.
We are interested in the dimensions of visible parts. From the definition it
follows easily that the orthogonal projection of $K$ onto $L$ equals that of the
visible part. Therefore, the Hausdorff, packing and box counting dimensions of a visible part are always between the correponding
dimensions of the set $K$ and its projection. In general, this
is all what can be said, that is, given any two numbers $t<s$ such that there
exists a set with dimension $s$ whose projection has dimension $t$ then, for
every $u\in[t,s]$, one can construct a set whose visible part has dimension $u$.
In order to obtain some nontrivial results, one has to study several visible
parts simultaneously. It follows from the Marstrand-Mattila projection theorem
\cite{Mar,Mat75} (see also \cite[Corollary 9.4]{Mat95}) that the Hausdorff
dimension of a visible part equals that of the set $K$ for almost all visible
parts provided the Hausdorff dimension of $K$ is at most $d-1$ and almost all
visible parts are at least $(d-1)$-dimensional if the Hausdorff dimension of $K$
is larger than $d-1$. This was shown in \cite{JJMO}, where a more general result
concerning $k$-planes was proved. 

According to the well-known visible part conjecture (see
\cite[Problem 11]{Mat04}) almost all visible parts are $(d-1)$-dimensional
provided the set has Hausdorff dimension at least $d-1$. An example of Davies and Fast \cite{DF} shows that the measure theoretic notation ``almost all'' cannot be replaced by a
topological typicality, that is, there is a compact 2-dimensional set in the
plane whose visible part is the whole set from a dense $G_\delta$-set of lines
$L$. Orponen \cite{O14}, in turn, proved in the plane that this phenomenon,
where the visible part is the whole set $K$, may happen only for a set of
directions whose Hausdorff dimension is at most $2-\dimH(K)$, where $\dimH$
stands for the Hausdorff dimension. In \cite{JJMO}, the visible part conjecture was proved for quasicircles, graphs of continuous functions and some
special examples of self-similar sets. In \cite{AJJRS}, it was shown to be
almost surely valid for the fractal percolation. Falconer and Fraser \cite{FF}
proved that if $K\subset \R^2$ is a self-similar set satisfying the convex open
set condition and if the projection of $K$ is an interval for all lines, then
the Hausdorff and box counting dimensions of all visible parts are equal to 1.
In the case when there are no rotations in the defining iterated function
system of $K$, they showed that if the projection of $K$ onto a line $L$ is an
interval, then the Hausdorff and box counting dimensions of the visible part
from $L$ are equal to 1. Rossi \cite{R} proved the visibility conjecture for
self-affine sets satisfying a projection condition and a property called strong
cone separation. As far as we know, these are the only cases, where the
visible part conjecture has been proved. However, there are several results
which show that a typical visible part is strictly smaller than the set itself.
A result of \cite{JJN} states that almost all visible parts of a compact
$s$-set, that is, a compact set having positive and finite $s$-dimensional
Hausdorff measure, have zero $s$-dimensional Hausdorff measure 
provided $s>d-1$. O'Neil
\cite{ON} studied planar continua and proved an explicit nontrivial upper bound
for the Hausdorff dimension of almost all visible parts from points, that is,
from 0-planes. Quite recently, Orponen \cite{O20} proved a general theorem
according to which the Hausdorff dimensions of almost all visible parts are at
most $d-\frac 1{50d}$ for all compact sets in $\R^d$. In particular,
if the Hausdorff dimension of a set in $\R^d$ is larger than $d-\frac 1{50d}$,
then the Hausdorff dimensions of almost all visible parts are strictly smaller
than the Hausdorff dimension of the set itself.

In this paper, we study dimensions of visible parts of self-similar sets $K$
generated by iterated function systems having finite rotation groups and
satisfying the weak separation condition. In Theorem~\ref{main}, we prove an upper
bound for the Assouad dimension of visible parts, which depends on the
direction and on the penetrable part of the set (for definition, see \eqref{penetrable}). As a
corollary, we obtain a criterion which guarantees the validity of the visible
part conjecture for a fixed hyperplane $L$, see Corollary~\ref{corollary}. In
particular, for planar self-similar sets generated by iterated function systems
with finite rotation groups and satisfying the weak separation condition, the visible
part from a line $L$ is 1-dimensional provided the projection of $K$ onto $L$
is a finite or countable union of intervals.  Under the above assumptions, we
also show that the Assouad dimension of a visible part is strictly smaller than
the Hausdorff dimension of the set provided the projection contains interior
points and the Hausdorff dimension of $K$ is strictly larger than $d-1$, see
Theorem~\ref{theorem2}. In the case of trivial rotation
group, we prove that if the projection satisfies the weak separation condition,
then the upper box counting dimension of the visible part equals the Hausdorff
dimension of the projection, see Proposition~\ref{WSC}. Finally,
in Section~\ref{examples}, we give examples and discuss the role of our
assumptions.

\section{Results}\label{results}

Fix a finite set $\Lambda$ and an integer $d\in\N\setminus\{1\}$, where
$\N:=\{1,2,\dots\}$. Set $\Lambda^*:=\bigcup_{n=0}^\infty\Lambda^n$, where
$\Lambda^0$ is the empty word. Let $\cF:=\{f_i(x):=r_iO_i(x)+t_i\}_{i\in \Lambda}$
be an iterated function system on $\R^d$, where $0<r_i<1$, $O_i$ is an element
of the orthogonal group $\textrm{O}(d,\R)$ and $t_i\in\R^d$ for all
$i\in\Lambda$. Denote by $K$ the attractor of $\cF$, that is, $K$ is the unique
nonempty compact set in $\R^d$ satisfying 
\[
 K=\bigcup_{i\in \Lambda}f_i(K).
\]
Since all maps in $\cF$ are similarity transformations, $K$ is said to be a
\emph{self-similar set}. Equip $\Lambda^\N$ with the metric
\[
 d(\bi,\bj):=\prod_{n=1}^{|\bi\wedge\bj|}r_{i_n},
\]
where $\bi\wedge\bj$ is the longest common prefix of the sequences $\bi$ and
$\bj$ and $|u|$ is the length of a word $u\in\Lambda^*$. Note that this metric
generates the product topology on $\Lambda^\N$. For $u\in\Lambda^*$, we write
$f_u:=f_{u_{1}}\circ\dots\circ f_{u_{|u|}}$ and define the cylinder
\[
 [u]:=\{i\in\Lambda^\N\mid i|_{|u|}=u\},
\]
where $i|_n:=i_1\dots i_n$ for $n\in\N$. Let $\pi\colon\Lambda^\N\to K$
be the coding map associated to $\cF$, that is, for all $\bi\in\Lambda^\N$, we
define
\[
 \pi(\bi):=\lim_{n\to\infty}f_{i_1}\circ f_{i_2}\circ\cdots\circ f_{i_n}(0).
\]
The \emph{rotation group} of $\cF$, denoted by $G(\cF)$, is the closure of the
subgroup of $\textrm{O}(d,\R)$ generated by $\{O_i\mid i\in\Lambda\}$.

\begin{definition}
A self-similar set $K$ generated by an iterated function system
$\cF:=\{f_i\}_{i\in\Lambda}$ satisfies the \emph{weak separation condition} if
the identity is not an accumulation point of the set
$\{f_u^{-1}\circ f_v\mid u,v\in\Lambda^*, u\ne v\}$ in the space of similarities
equipped with the topology of pointwise convergence. 
\end{definition}

\begin{remark}\label{dimpreserves}
(a) We note that there are several different versions of the weak
separation condition, most of them being equivalent, see \cite{BG,LN,Z}.
  
(b) For all $z\in\R^d$ and all $r>0$, let
$Q(z,r):=r[0,1]^d+z-\frac r2(1,\dots,1)$ be the closed cube centred at $z$ and
with side length $r$. We record here a useful fact \cite[Theorem 1 (4a)]{Z}
concerning iterated function systems satisfying the weak separation
condition: there exist $M\in\N$ and $0<c<1$ such that, for all
$z\in K$ and all $r>0$, there exist $u_1,u_2,\ldots,u_M\in\Lambda^*$ satisfying
$cr\le\diam([u_j])\le r$ for all $j=1,\dots,M$ and
\[
 Q(z,r)\cap K\subset\bigcup_{j=1}^Mf_{u_j}(K).
\]

\end{remark}

For $z\in\R^d\setminus\{0\}$, let $\langle z\rangle$ be the line spanned by $z$.
For $\theta\in S^{d-1}$ and $y\in\langle\theta\rangle$, let $L_\theta$ be the
linear subspace perpendicular to $\theta$ and set $L_{\theta,y}:=y+L_\theta$.
Denote by $P_{\theta,y}$ the projection from $\R^d$ onto $L_{\theta,y}$, that is,
$P_{\theta,y}(z)$ is the point of $L_{\theta,y}$ that is closest to $z$. In the case
$y=0$, we simply write $P_\theta$. Given a compact set $F\subset\R^d$,
a direction $\theta\in S^{d-1}$ and a point $y\in\langle\theta\rangle$ such that
$L_{\theta,y}\cap F=\emptyset$, the \emph{visible part of $F$ from $L_{\theta,y}$}
is the set
\begin{equation}\label{visibledef}
  V_{\theta,y}(F):=\{z\in F\mid [z,P_{\theta,y}(z)]\cap F=\{z\}\},
\end{equation}
where $[z,z']$ is the line segment connecting the points $z,z'\in\R^d$.
Heuristically, $V_{\theta,y}(F)$ consists of those points of $F$ which one can
see from the hyperplane $L_{\theta,y}$ while looking out perpendicular to the
plane. For a more general definition of visible parts, see \cite{JJMO}.

For a set $F\subset\R^d$, denote its interior by $\Int(F)$. For
every $\theta\in S^{d-1}$, we define the
\emph{penetrable part of $F$ in direction $\theta$} by setting
\begin{equation}\label{penetrable}
 \Pen_\theta(F):=F\cap P_\theta^{-1}\bigl(P_\theta(F)\setminus\Int(P_\theta(F))\bigr).
\end{equation}

\begin{remark}\label{cantorval}
In the case $G(\cF)=\{\Id\}$, the projection of $K$ is a self-similar set. We
remark that there are self-similar sets with nonempty interior, whose boundary
is a Cantor set (see e.g. \cite{BFGPS}). A planar self-similar set with
an interesting projectional structure is presented in Example~\ref{Meng}. 
\end{remark}

We denote by $\dimA$ the Assouad dimension, by $\dimB$ the box counting
dimension and by $\udimB$ the upper box counting dimension. Recall that
$\dimH(F)\le\udimB(F)\le\dimA(F)$ for all bounded sets $F\subset\R^d$. We are
now ready to state our main theorem. 

\begin{theorem}\label{main}
Let $K\subset\R^d$ be a self-similar set generated by an iterated function
system $\cF$ with a finite rotation group $G(\cF)$ and satisfying the  weak separation condition. Then, for all $\theta\in S^{d-1}$ and $y\in\langle\theta\rangle$ with
$L_{\theta,y}\cap K=\emptyset$, we have that 
\begin{equation}\label{eq:main}
  \dimA(V_{\theta,y}(K))\le\max\{d-1,\max_{g\in G(\cF)}\dimH(\Pen_{g(\theta)}(K))\}.
\end{equation}
\end{theorem}

\begin{remark}
We note that the statement of Theorem~\ref{main} is general.
However, it gives a useful bound for $\dimA(V_{\theta,y}(K))$ only when
$\Int(P_{g(\theta)}(K))\neq \emptyset$ for all $g\in G(\cF)$, for if
$\Int(P_{g(\theta)}(K))=\emptyset$ for some $g\in G(\cF)$, then
$\Pen_{g(\theta)}(K)=K$. Recall that $\dimH(K)=\dimA(K)$ by the weak
separation condition (see \cite[Theorem 1.4]{FHOR}).
\end{remark}

We have an immediate corollary of Theorem~\ref{main}.

\begin{corollary}\label{corollary}
Let $K\subset\R^d$ be a self-similar set generated by an iterated function
system $\cF$ with a finite rotation group $G(\cF)$ and satisfying
the  weak separation condition.   Fix $\theta\in S^{d-1}$. If
$\dimH(P_{g(\theta)}(\Pen_{g(\theta)}(K)))\le d-2$ for all $g\in G(\cF)$, then  
\[
 \dimA(V_{\theta,y}(K))\le d-1
\]
for all $y\in\langle\theta\rangle$ with $L_{\theta,y}\cap K=\emptyset$. In
particular, if $d=2$ and if $P_{g(\theta)}(K)$ is a finite or countable union of
intervals for every $g\in G(\cF)$, then 
\[
 \dimH(V_{\theta,y}(K))=\dimA(V_{\theta,y}(K))=1.
\]
\end{corollary}

\begin{proof}
The claim follows since $\dimH(A)\le\dimH(P_{\theta,y}(A))+1$ for all sets
$A\subset\R^d$ (see e.g. \cite[Theorem 8.10]{Mat95}).
\end{proof}

\begin{remark}\label{OSCremark}
Recall that Falconer and Fraser \cite{FF} proved that all visible parts of a
self-similar set have Hausdorff and box counting dimension equal to 1 provided
that the defining iterated function system satisfies the convex open set
condition and all projections of the self-similar set are intervals. These
assumptions play a crucial role in \cite{FF}. Corollary~\ref{corollary} shows
that the assumption of convexity of the open set is not needed, at least in the
case of a finite rotation group. Further, the condition that every projection is
a single interval may be relaxed to a finite or countable union of intervals.
We remind that our method is completely different from that of \cite{FF}. Also
Rossi's work \cite{R} on visible parts of self-affine sets relies essentially on
the assumption that the projection of the attractor is a single interval for
many directions. Furthermore, the strong cone separation condition used in
\cite{R} is never valid for self-similar sets.
\end{remark}

Before the proof of Theorem~\ref{main}, we introduce the Furstenberg's dimension
conservation principle for compact sets (see \cite{Fu08}), which is a key
ingredient in our proof. For a set $E\subset\R^d$ and $\theta\in S^{d-1}$, we say
that the projection $P_\theta\colon\R^d\to L_\theta$ is
\emph{dimension conserving} for $E$ if for some $\delta\ge 0$,
\begin{equation}\label{DC}
  \delta+\dimH\left(\left\{x\in L_\theta\mid\dimH\left(P_\theta^{-1}(x)
  \cap E\right)\ge\delta\right\}\right)\ge\dimH(E).
\end{equation}
Here we adopt the convention that the dimension of the empty set is $-\infty$. 
Observe that one could replace the latter $\ge$-sign in \eqref{DC} by equality,
since the left hand side can never be strictly larger than the right hand side
by a classical theorem of Marstrand \cite{Mar2}, (see also
\cite[Corollary 2.10.27]{Fed}), but we do not need that information in our
paper.
Let $E$ be a closed subset of $[0,1]^d$. A closed set $A\subset [0,1]^d$ is a
\emph{mini-set} of $E$ if $A\subset (\lambda E+x)\cap [0,1]^d$ for some scalar
$\lambda\ge 1$ and $x\in\R^d$. A closed set $B\subset [0,1]^d$ is a
\emph{micro-set} of $E$ if there is a sequence $(A_n)$ of mini-sets of $E$ with
$\lim_{n\to\infty}A_n=B$ in the Hausdorff metric. 

The following result is a special case of \cite[Theorem 6.1]{Fu08}. 

\begin{theorem}\label{thm-DC}
For every compact set $E\subset \R^d$ and for every $\theta\in S^{d-1}$, there is
a micro-set $B$ of $E$ with $\dimH(B)\ge\dimA(E)$ for which $P_\theta$ is
dimension conserving.   
\end{theorem}

\begin{remark}\label{assouaddim}
In the statement of \cite[Theorem 6.1]{Fu08}, there is $\dimH(E)$ instead of
$\dimA(E)$. However, the proof of \cite[Theorem 6.1]{Fu08} works for the
Assouad dimension as is evident from the remarks before
\cite[Proposition 6.1]{Fu08}. Recall also \cite[Theorem 2.4]{FHKY}, which states
that the Assouad dimension of a set equals the dimension of the gallery of its
micro-sets.
\end{remark}  

In the next lemma, we state some basic properties needed in the proof of
Theorem~\ref{main}. Recall that a homothety is a map $h\colon\R^d\to\R^d$,
where $h(x):=r(x-t)+t$ with $r\in\R\setminus\{0\}$ and $t\in\R^d$.

\begin{lemma}\label{basic}
Fix $\theta\in S^{d-1}$, $y\in\R^d$, $g\in\textrm{O}(d,\R)$ and a homothety $h$.
Let $F\subset\R^d$ be a nonempty compact set. Then
\begin{enumerate}
  \item $h(V_{\theta,y}(F))=V_{\theta,h(y)}(h(F))$, \label{basica} 
  \item $g\circ h(F)$ and $h\circ g(F)$ differ only by a translation and
    \label{basicb}
  \item $V_{\theta,y}(g(F))=g(V_{g^{-1}(\theta),g^{-1}(y)}(F))$. \label{basicc}
\end{enumerate}
\end{lemma}

\begin{proof}
The claims follow from straightforward calculations.
\end{proof}  

We are now ready for the proof of Theorem~\ref{main}.

\begin{proof}[Proof of Theorem \ref{main}]
Fix $\theta\in S^{d-1}$ and $y\in\langle\theta\rangle$ such that
$L_{\theta,y}\cap K=\emptyset$. We will only deal with the case, where $K$ is
included in one of the components of $\R^d\setminus L_{\theta,y}$. The general
situation can be deduced from this special case as follows: Since
$L_{\theta,y}\cap K=\emptyset$, there is $k\in\N$ such that, for all
$u\in\Lambda^k$, the set $f_u(K)$ is included in one of the components of
$\R^d\setminus L_{\theta,y}$. 
Since $f_u$ is a similarity, we can write $f_u(K)$ as a composition
\begin{equation}\label{rotahomo}
  f_u(K)=g_u\circ h_u(K),
\end{equation}  
where $g_u\in G(\cF)$ is the orthogonal component of $f_u$ and $h_u$ is a homothety. 

Notice also that $f_u(K)$ is a self-similar set generated by the iterated
function system $\cF_u:=\{f_u\circ f_i\circ f_u^{-1}\mid i\in\Lambda\}$  that also satisfies the weak separation condition.     Clearly,
$G(\cF_u)=G(\cF)$ and
\[
 V_{\theta,y}(K)\subset\bigcup_{u\in\Lambda^k}V_{\theta,y}(f_u(K)).
\]
Assuming that the special case has been proved, one concludes by
Lemma~\ref{basic} that $\dimA(V_{\theta,y}(f_u(K)))$ is bounded by the upper bound
given in \eqref{eq:main}. Thus, the stability of the Assouad dimension under
finite unions implies that \eqref{eq:main} holds in the general situation.
 
Let $E=\overline{V_{\theta,y}(K)}$ be the closure of $V_{\theta,y}(K)$. By
Remark~\ref{dimpreserves}, for every $x\in E$ and $r>0$, there exist
$u_1,u_2,\ldots,u_M\in\Lambda^*$ satisfying $cr\le\diam([u_j])\le r$
for all $j=1,\dots,M$ such that
\[
 Q(x,r)\cap E\subset\bigcup_{j=1}^Mf_{u_j}(K)\cap E.
\]
Observe that $V_{\theta,y}(K)\cap f_{u_j}(K)\subset V_{\theta,y}(f_{u_j}(K))$ for
every $j\in\{1,\dots,M\}$. Therefore, by \eqref{rotahomo},
\begin{equation}\label{localsubset}
  Q(x,r)\cap E\subset\bigcup_{j=1}^M\overline{V_{\theta,y}(f_{u_j}(K))}
  =\bigcup_{j=1}^M\overline{V_{\theta,y}(g_{u_j}\circ h_{u_j}(K))}.
\end{equation}
If $A$ is a mini-set of $E$, then $A\subset h(Q(x,r)\cap E)$ for some
$x\in [0,1]^d$ and $0<r\le 1$, where $h$ is a homothety with ratio $r^{-1}$.
Combining \eqref{localsubset} and Lemma~\ref{basic}, we conclude that
\begin{equation}\label{visibleminiset}
  A\subset\bigcup_{j=1}^M\overline{V_{\theta,y'}(g_j\circ h_j(K))},
\end{equation}
where $g_j\in G(\cF)$, $y'\in\R^d$ and the ratio of the homothety $h_j$
is between $c$ and $1$ for all $j=1,\dots,M$. 
Let $B$ be a micro-set of $E$. Then $B=\lim_{n\to\infty} A_n$ for a sequence
$(A_n)$ of mini-sets of $E$. Applying \eqref{visibleminiset} to $A_n$ for all
$n\in\N$, we find a sequence of $M$-tuples of sets
$(g_1^n\circ h_1^n(K),\dots,g_M^n\circ h_M^n(K))$. Since $G(\cF)$
is finite, there is a constant subsequence $(g_1^{n_k},\dots,g_M^{n_k})$. Since
the ratios and translations of the homotheties $h_j^{n_k}$ are bounded, we
may choose a further subsequence such that
$(h_1^{n_{k_l}},\dots,h_M^{n_{k_l}})$ converges. Using
Lemma~\ref{basic}, we conclude that
\eqref{visibleminiset} is valid with a mini-set $A$ replaced by a micro-set $B$.

Now we apply Theorem~\ref{thm-DC} to $E$ in order to find a micro-set $B$ of
$E$ whose Hausdorff dimension is at least $\dimA(E)$ and for which $P_\theta$ is
dimension conserving. Applying \eqref{visibleminiset} to $B$, we conclude that
there exist $\delta\ge 0$ and $F:=\overline{V_{\theta,y'}(g_j\circ h_j(K))}$
such that
\begin{equation}\label{applyDC}
  \delta+\dimH\left\{x\in P_\theta(F)\mid\dimH\left(P_\theta^{-1}(x)\cap F\right)
  \ge\delta\right\}\ge\dimA(E).
\end{equation}
Let $\tau:=g_j^{-1}(\theta)$. Combining Lemma~\ref{basic} with \eqref{applyDC},
we deduce that
\begin{equation}\label{applyDC2}
  \delta+\dimH\left\{x\in P_\tau\left(\overline{V_{\tau,z}(K)}\right)\mid
  \dimH\left(P_\tau^{-1}(x)\cap \overline{V_{\tau,z}\left(K\right)}\right)
  \ge\delta\right\}\ge\dimA(E)
\end{equation}
for some $z\in\R^d$, which we may choose such that $K$ is included in one of the
components of $\R^d\setminus L_{\tau,z}$. Clearly, we may assume that
$\dimA(E)>d-1$. In view of \eqref{applyDC2}, the constant $\delta$ must be
strictly positive, since
$\dimH\left(P_\tau\left(\overline{V_{\tau,z}(K)}\right)\right)\le d-1$. Let 
\[
 D:=\left\{x\in P_\tau\left(\overline{V_{\tau,z}(K)}\right)\mid
 \dimH\left(P_\tau^{-1}(x)\cap\overline{V_{\tau,z}(K)}\right)\ge\delta\right\}.
\]
Since $\delta>0$, we can find, for every $\vep>0$, a nonempty subset
$D_\vep\subset D$ and $\gamma>0$ such that
\begin{equation}\label{projectionbound}
  \dimH(D_\vep)\ge\dimH(D)-\vep
\end{equation}
and, for every $x\in D_\vep$, there are
$C_x\subset P_\tau^{-1}(x)\cap\overline{V_{\tau,z}\left(K\right)}$ and
$w\in P_\tau^{-1}(x)\cap\overline{V_{\tau,z}(K)}$ with 
\begin{equation}\label{sectionbound}
  \dimH(C_x)\ge \delta-\vep
\end{equation}
and
\begin{equation}\label{positivedistance}
  \dist(w,L_{\tau,z})\ge \dist(w',L_{\tau,z})+\gamma\text{ for every }w'\in C_x.
\end{equation}

Fix such an $n_0\in\N$ that, for every $u\in\Lambda^{n_0}$, we have that
$\diam(f_u(K))\le\frac\gamma 2$. Let $x\in D_\vep$ and $w'\in C_x$. Since $K$ is
closed, $\overline{V_{\tau,z}(K)}\subset K$. Therefore, $w'\in f_u(K)$ for some
$u\in\Lambda^{n_0}$. If $P_\tau(w')\in\Int(P_\tau(f_u(K)))$, there is
no point $y\in\overline{V_{\tau,z}(K)}\cap P_\tau^{-1}(P_\tau(w'))$ with
$\dist(y,L_{\tau,z})>\dist(w',L_{\tau,z})+\diam(f_u(K))$ (recall that $K$ is
included in one of the components of $\R^d\setminus L_{\tau,z}$).
Inequality~\eqref{positivedistance} and the fact that
$w\in\overline{V_{\tau,z}(K)}$ imply that $w'\in\Pen_\tau(f_u(K))$. From
this we deduce that 
\begin{equation}\label{Cantorinclusion}
  \bigcup_{x\in D_\vep}\{C_x\}\subset\bigcup_{u\in\Lambda^{n_0}}\Pen_\tau(f_u(K)).
\end{equation}
Applying the classical Marstrand theorem \cite{Mar2} and
using the dimension bounds in \eqref{projectionbound} and
\eqref{sectionbound} together with \eqref{applyDC2}, we conclude that 
\[
 \dimH\left(\bigcup_{x\in D_\vep}\{C_x\}\right)\ge\dimH(D)+\delta-2\vep
 \ge\dimA(E)-2\vep.
\]
On the other hand, note that the Hausdorff dimension of the set on the right
hand side of \eqref{Cantorinclusion} is bounded above by
\[
 \max_{g\in G(\cF)}\dimH(\Pen_{g(\theta)}(K)).
\]
Since $\vep$ is arbitrary, we get that $\dimA(E)$ is bounded by the above
quantity, which is the desired conclusion.
\end{proof}

Our second theorem states that if the projection of a self-similar set with a
finite rotation group contains interior points, the dimension of the
visible part is strictly smaller than the dimension of the set itself. 

\begin{theorem}\label{theorem2}
Let $K\subset\R^d$ be a self-similar set generated by an iterated function
system $\cF$ with a finite rotation group $G(\cF)$ and satisfying the  weak separation condition. Assume that $\dimH(K)>d-1$. Fix $\theta\in S^{d-1}$. Suppose that
$\Int(P_\theta(K))\ne\emptyset$. Then $\dimA(V_{\theta,y}(K))<\dimH(K)$ for all
$y\in\langle\theta\rangle$ with $L_{\theta,y}\cap K=\emptyset$.
\end{theorem}

\begin{proof}
By Theorem~\ref{main}, it is enough to show that
\begin{equation}\label{boxbound}
  \max_{g\in G(\cF)}\dimH(\Pen_{g(\theta)}(K))<\dimH(K).
\end{equation}
We do that using a porosity argument. Recall that a
set $A$ in a metric space $Z$ is porous in $Z$ if there exists a constant $c>0$
so that, for every ball $B(z,r)$ in $Z$, there exists a ball
$B(z_1,cr)\subset B(z,r)$ satisfying $B(z_1,cr)\cap A=\emptyset$.
It is well known (see \cite[Lemma 3.12]{BHR}) that if the space
$Z$ is Ahlfors regular, the inequality $\dimA A\le \dimH Z-\delta$ is valid for
every porous set $A$ in $Z$, where $\delta>0$ is a constant depending only on
$Z$ and the parameter $c$ involved in the definition of porosity of $A$ in $Z$.
Since $\dimH(K)>d-1$, the attractor $K$ is not contained in a
hyperplane. Further, as $\cF$ satisfies the weak separation condition, the
self-similar set $K$ is Ahlfors regular by \cite[Theorem 2.1]{FHOR}. Thus, to
prove \eqref{boxbound}, we only need to show that, for every $g\in G(\cF)$, the
set $\Pen_{g(\theta)}(K)$ is porous in $K$.

To that end, observe first that, for all $g\in G(\cF)$, 
\begin{equation}\label{intinvariant}
  \begin{split}
    &\text{if }P_{g(\theta)}(x)\in\Int(P_{g(\theta)}(K))\text{ for }x\in K,\\
    &\text{then }P_{O_i\circ g(\theta)}(f_i(x))\in\Int(P_{O_i\circ g(\theta)}(K))
    \text{ for all }f_i\in\cF.
  \end{split}
\end{equation}
As $\Int(P_\theta(K))\ne\emptyset$, we have that
$\Int(P_{g(\theta)}(K))\ne\emptyset$ for all $g\in G(\cF)$ by \eqref{intinvariant}
and the finiteness of $G(\cF)$. Further, there are $r>0$ and points
$x_g\in P_{g(\theta)}(K)$ for all $g\in G(\cF)$ such that
$B(x_g,r)\subset\Int(P_{g(\theta)}(K))$. Thus, for
all $g\in G(\cF)$, one can find a finite word $u_g\in\Lambda^*$ so that
$P_{g(\theta)}(f_{u_g}(K))\subset B(x_g,r)$. From this we infer that 
\begin{equation}\label{eq:proof-theorem2-1}
  P_ {g(\theta)}\Bigl(\Conv\bigl(f_{vu_{O_v^{-1}\circ g}}(K)\bigr)\Bigr)
  \subset\Int(P_{g(\theta)}(K))
\end{equation}
for all $g\in G(\cF)$ and for all finite words $v\in\Lambda^*$, where
$f_v(x)=r_vO_v(x)+t_v$ is the decomposition of $f_v$ and $\Conv(A)$
denotes the convex hull of a set $A$.
The inclusion \eqref{eq:proof-theorem2-1} clearly implies that
\begin{equation}\label{eq:proof-theorem2-2}
 \Conv\bigl(f_{vu_{O_v^{-1}\circ g}}(K)\bigr)\cap\Pen_{g(\theta)}(K)=\emptyset.
\end{equation}
Since $G(\cF)$ is finite, $n_0:=\max\{|u_g|\mid g\in G(\cF)\}<\infty$.  Using
this fact and \eqref{eq:proof-theorem2-2}, it is now readily checked that
$\Pen_{g(\theta)}(K)$ is porous in $K$.
\end{proof}  

We complete this section with an observation connecting the weak separation
condition of projections to the dimensions of visible parts. 

\begin{proposition}\label{WSC}
Let $K\subset\R^d$ be a self-similar set generated by an iterated function
system $\cF:=\{f_i\}_{i\in \Lambda}$ with $G(\cF)=\{\Id\}$. Suppose that, for some
$\theta\in S^1$, the projection $P_\theta(K)$ satisfies the weak separation
condition. Then, for all $y\in\langle\theta\rangle$ with
$L_{\theta,y}\cap K=\emptyset$, we have that 
\[
\udimB(V_{\theta,y}(K))=\dimH(P_\theta(K)).
\]
\end{proposition}

\begin{proof}
As in the proof of Theorem~\ref{main}, we may assume that $K$ is included in one
of the components of $\R^d\setminus L_{\theta,y}$. Fix $y\in\langle\theta\rangle$
with $L_{\theta,y}\cap K=\emptyset$. Recall that the Hausdorff and box counting
dimensions are equal for every self-similar set, see \cite[Corollary 3.3]{F97}.
The fact that $G(\cF)=\{\Id\}$ implies that $(P_\theta\circ f)_u=P_\theta\circ f_u$
for all $u\in\Lambda^*$ and $P_\theta(K)$ is a self-similar set generated by the
iterated function system $\{P_\theta\circ f_i\}_{i\in \Lambda}$. Thus,
$\dimH(P_\theta(K))=\dimB(P_\theta(K))$. Let
$\rmin:=\min\{r_i\mid i\in\Lambda\}$. If the claim is not true, then
$\udimB(V_{\theta,y}(K))>\dimB(P_\theta(K))$. Therefore, for every $M\ge 1$, there are
arbitrarily small numbers $\delta>0$ and sets $I_{\delta,M}\subset\Lambda^*$ with
cardinality at least $M$ such that the following properties hold: for all
$u\in I_{\delta,M}$,
\[
f_u(K)\cap V_{\theta,y}(K)\ne\emptyset\text{ and }
\delta\le\diam(f_u(K))\le\rmin^{-1}\delta
\]
and, moreover, for all $u,v\in I_{\delta,M}$ with $u\ne v$, 
\[
\dist(f_u(K),f_v(K))>\diam(f_u(K))\text{ and }
\dH\bigl(P_\theta\circ f_u(K),P_\theta\circ f_v(K)\bigr)\le\delta,
\]
where the Hausdorff metric is denoted by $\dH$.

Now all the sets $P_\theta\circ f_u(K)$ are included in a
ball, whose radius is comparable to $\delta$. Further, they are
homothetic copies of each other, since $G(\cF)=\{\Id\}$, and their diameters are
comparable to $\delta$. Therefore, there exist $\vep:=\vep(M,\rmin,d)$
with $\lim_{M\to\infty}\vep=0$ and $u,v\in I_{\delta,M}$ such that 
\[
|P_\theta\circ f_u(z)-P_\theta\circ f_v(z)|\le\vep\diam(f_u(K))\text{ for all }
  z\in K,
\]
that is, there exists a sequence
$\bigl((P_\theta\circ f_{u_n})^{-1}\circ(P_\theta\circ f_{v_n})\bigr)_{n\in\N}$
converging to $\Id$. Under the weak separation condition for $P_\theta(K)$, there
exist $u\ne v$ such that $P_\theta\circ f_u=P_\theta\circ f_v$. Since
$\dist(f_u(K),f_v(K))>\diam(f_u(K))$ and $f_u(K)$ and $f_v(K)$ are included in
the same component of $\R^d\setminus L_{\theta,y}$, it follows that either
$f_u(K)\cap V_{\theta,y}(K)=\emptyset$ or $f_v(K)\cap V_{\theta,y}(K)=\emptyset$, which is
a contradiction.
\end{proof}

\section{Examples}\label{examples}

In this section, we present some examples illustrating the role of assumptions
in our theorems. One may ask, if it is necessary to take the maximum over the
dimensions of the penetrable parts over the orbit of $\theta$ under the
rotation group in Theorem~\ref{main} or if the dimension is constant along this
orbit. We show that the dimension of the penetrable part is not necessarily
constant. Recall the the open set condition implies the weak
separation condition.

\begin{example}\label{Meng}
Let $\cF:=\{f_1,\dots,f_5\}$ be an iterated function system on $\R^2$, where
$\{f_1,\dots,f_4\}$ generates the standard four corner Cantor set with
contraction ratio $\lambda\ge\frac 13$ and $f_5$ contracts by
$\rho<\frac{1-2\lambda}{\sqrt 2}$, rotates by an
angle $\frac{\pi}4$ and has the centre of $[0,1]^2$ as the fixed point, see
Figure~\ref{Mengfig}.
\begin{figure}
\begin{center}
\begin{tikzpicture}[xscale=2,yscale=2]
\draw[dashed] (0,0) -- (3,0) -- (3,3) -- (0,3) -- (0,0);
\draw[-] (0,0) -- (1,0) -- (1,1) -- (0,1) -- (0,0);
\draw[-] (0,2) -- (1,2) -- (1,3) -- (0,3) -- (0,2);
\draw[-] (2,0) -- (3,0) -- (3,1) -- (2,1) -- (2,0);
\draw[-] (2,2) -- (3,2) -- (3,3) -- (2,3) -- (2,2);
\draw[-] (1.5,1.3) -- (1.7,1.5) -- (1.5,1.7) -- (1.3,1.5) -- (1.5,1.3);
\draw[] (0.5,0.5)  node {$1$};
\draw[] (2.5,0.5)  node {$2$};
\draw[] (2.5,2.5)  node {$3$};
\draw[] (0.5,2.5)  node {$4$};
\draw[] (1.5,1.5)  node {$5$};
\end{tikzpicture}
\caption{Images of the unit square under the maps $f_1,\dots,f_5$ in
  Example~\ref{Meng}.}\label{Mengfig}
\end{center}
\end{figure}
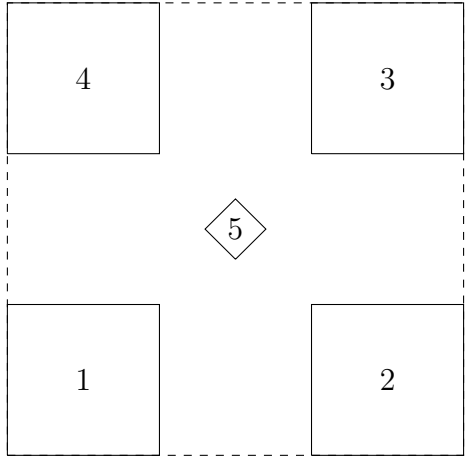
Then $\cF$ satisfies the open set condition and
$G(\cF)=\{O(\frac{k\pi}4)\mid k=0,\dots,7\}$, where $O(\alpha)$ is the rotation
by angle $\alpha$. Let $K$ be the attractor of $\cF$. Then
$P_{O(\frac{k\pi}4)(1,0)}(K)$ is an interval for odd $k$ and, for even $k$,
it is the union of the $\lambda$-Cantor set on the line and the
countable union of intervals, which are the projections of the squares
$f_u\circ f_5([0,1]^2)$, where $u\in\{1,2,3,4\}^*$. Observe that the
$\lambda$-Cantor set consists of the limit points of the end points of the
intervals. In particular, $\Pen_{O(\frac{(2k+1)\pi}4)(1,0)}(K)$ consists of two
opposite corner points of $[0,1]^2$, while $\Pen_{O(\frac{2k\pi}4)(1,0)}(K)$ is the
union of the standard four corner Cantor set generated by $\{f_1,\dots,f_4\}$
and the left and right endpoints of the squares $f_u\circ f_5([0,1]^2)$, where
$u\in\{1,2,3,4\}^*$.
\end{example}

Next example shows that there are self-similar sets satisfying the open set
condition such that their projections are finite unions of intervals. This
property of the following example must be well know, but we could not find a
reference.

\begin{example}\label{fourcorner}
Consider the standard four corner Cantor set $K$ in $\R^2$ with contraction
ratio $\lambda\ge\frac 13$. Then, for all 
$\theta\in S^1\setminus\{\pm(1,0),\pm (0,1)\}$, the
projection $P_\theta(K)$ is a finite union of intervals. We describe them in the
case $\lambda=\frac 13$. Denote by $\alpha_\theta$ the angle which
$\theta\in S^1$ makes with the positive $x$-axis. By symmetry, it is enough to
consider the case $\alpha_\theta\in[0,\frac\pi 4]$. For all $k\in\N$, if
$\tan\alpha_\theta\in\mathopen[3^{-k},3^{-(k-1)}\mathclose[$, then $P_\theta(K)$ is a
union of $2^{k-1}$ intervals. Indeed, denoting by $K_{k-1}$ the approximation of
$K$ at construction level $k-1$, the projection $P_\theta(K_{k-1})$ is
a union of $2^{k-1}$ intervals. At level $k$, the projection of every level $k-1$
construction square is the union of two intervals, but the gap between these
intervals is covered by the projection of another level $k-1$ construction
square due to the choice of $\alpha_\theta$ and the fact that
$1-2\lambda\le\lambda$, that is, the gap between the intervals is smaller than
the length of the intervals, see Figure~\ref{fourcornerfig}.
\begin{figure}
\begin{center}
\begin{tikzpicture}[xscale=2,yscale=2]
\draw[dashed] (0,0) -- (3,0) -- (3,3) -- (0,3) -- (0,0);
\draw[dashed] (0,0) -- (1,0) -- (1,1) -- (0,1) -- (0,0);
\draw[dashed] (0,2) -- (1,2) -- (1,3) -- (0,3) -- (0,2);
\draw[dashed] (2,0) -- (3,0) -- (3,1) -- (2,1) -- (2,0);
\draw[dashed] (2,2) -- (3,2) -- (3,3) -- (2,3) -- (2,2);
\draw[-] (0,0) -- (0.33,0) -- (0.33,0.33) -- (0,0.33) -- (0,0);
\draw[-] (0.67,0) -- (1,0) -- (1,0.33) -- (0.67,0.33) -- (0.67,0);
\draw[-] (0.67,0.67) -- (1,0.67) -- (1,1) -- (0.67,1) -- (0.67,0.67);
\draw[-] (0,0.67) -- (0.33,0.67) -- (0.33,1) -- (0,1) -- (0,0.67);
\draw[-] (2,0) -- (2.33,0) -- (2.33,0.33) -- (2,0.33) -- (2,0);
\draw[-] (2.67,0) -- (3,0) -- (3,0.33) -- (2.67,0.33) -- (2.67,0);
\draw[-] (2.67,0.67) -- (3,0.67) -- (3,1) -- (2.67,1) -- (2.67,0.67);
\draw[-] (2,0.67) -- (2.33,0.67) -- (2.33,1) -- (2,1) -- (2,0.67);
\draw[-] (2,2) -- (2.33,2) -- (2.33,2.33) -- (2,2.33) -- (2,2);
\draw[-] (2.67,2) -- (3,2) -- (3,2.33) -- (2.67,2.33) -- (2.67,2);
\draw[-] (2.67,2.67) -- (3,2.67) -- (3,3) -- (2.67,3) -- (2.67,2.67);
\draw[-] (2,2.67) -- (2.33,2.67) -- (2.33,3) -- (2,3) -- (2,2.67);
\draw[-] (0,2) -- (0.33,2) -- (0.33,2.33) -- (0,2.33) -- (0,2);
\draw[-] (0.67,2) -- (1,2) -- (1,2.33) -- (0.67,2.33) -- (0.67,2);
\draw[-] (0.67,2.67) -- (1,2.67) -- (1,3) -- (0.67,3) -- (0.67,2.67);
\draw[-] (0,2.67) -- (0.33,2.67) -- (0.33,3) -- (0,3) -- (0,2.67);
\draw[-] (0,1) -- (3,1.8);
\draw[-] (2,0.33) -- (3,0.597);
\draw[-] (1,0) -- (3,0.533);
\draw[-] (0,0.33) -- (3,1.13);
\draw[] (0.5,0.5)  node {$1$};
\draw[] (2.5,0.5)  node {$2$};
\draw[] (2.5,2.5)  node {$3$};
\draw[] (0.5,2.5)  node {$4$};
\end{tikzpicture}
\caption{Projecting the four corner Cantor set.}\label{fourcornerfig}
\end{center}
\end{figure}
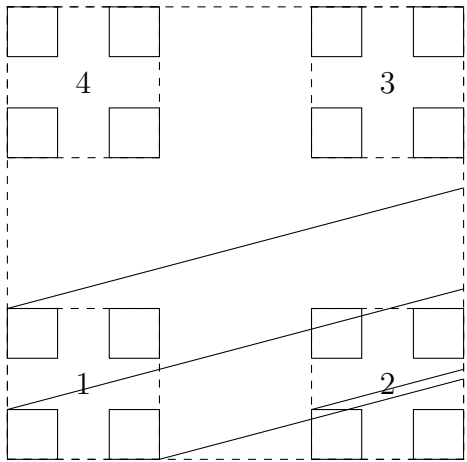
\end{example}

The projection of a self-similar set may be a countable union of
intervals. This is demonstrated by the following example.

\begin{example}\label{Ville}
Let $K$ be the attractor generated by an iterated function system
$\cF:=\{f_1,f_2,f_3\}$ on $[0,1]^2$, where $f_1$ and $f_2$ contract by the
factor $\frac 12$, $f_1$ translates by $(\frac 12,0)$, $f_2$ translates by
$(\frac 12,\frac 12)$ and $f_3$ contracts by $\frac 14$, rotates by an angle
$\frac{\pi}2$ and translates to the upper left corner. For an illustration of
the third construction step, see Figure~\ref{Villefig}. Now the projection of
$K$ onto the $y$-axis is the unit interval and onto the $x$-axis it is the union
$\bigcup_{k=0}^\infty [1-2^{-k},1-2^{-k}+2^{-k-2}]\cup\{1\}$.
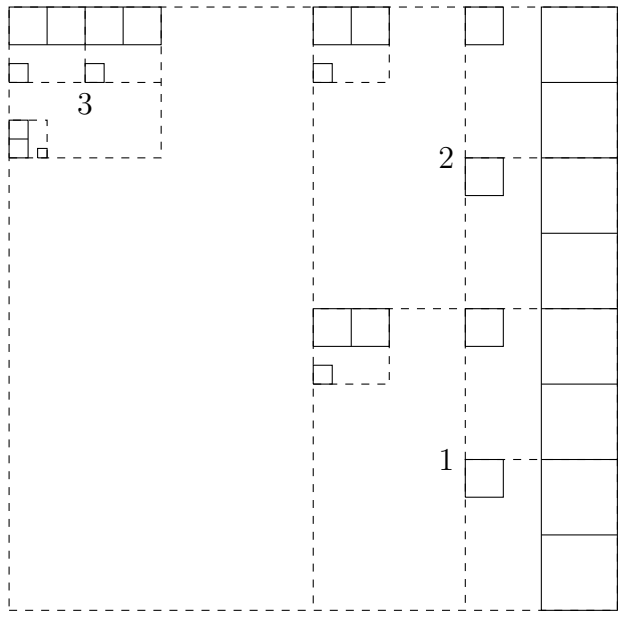
\begin{figure}
\begin{center}
\begin{tikzpicture}[xscale=2,yscale=2]
\draw[dashed] (0,0) -- (4,0) -- (4,4) -- (0,4) -- (0,0);
\draw[dashed] (4,2) -- (2,2) -- (2,0);
\draw[dashed] (2,4) -- (2,2);
\draw[dashed] (0,3) -- (1,3) -- (1,4);
\draw[dashed] (4,1) -- (3,1) -- (3,0);
\draw[dashed] (3,2) -- (3,1);
\draw[dashed] (2,1.5) -- (2.5,1.5) -- (2.5,2);
\draw[dashed] (4,3) -- (3,3) -- (3,2);
\draw[dashed] (3,4) -- (3,3);
\draw[dashed] (2,3.5) -- (2.5,3.5) -- (2.5,4);
\draw[dashed] (0,3.5) -- (1,3.5);
\draw[dashed] (0.5,3.5) -- (0.5,4);
\draw[dashed] (0.25,3) -- (0.25,3.25) -- (0,3.25);
\draw[-] (3.5,0) -- (4,0) -- (4,4) -- (3.5,4) -- (3.5,0);
\draw[-] (3.5,0.5) -- (4,0.5);
\draw[-] (3.5,1) -- (4,1);
\draw[-] (3.5,1.5) -- (4,1.5);
\draw[-] (3.5,2) -- (4,2);
\draw[-] (3.5,2.5) -- (4,2.5);
\draw[-] (3.5,3) -- (4,3);
\draw[-] (3.5,3.5) -- (4,3.5);
\draw[-] (3,0.75) -- (3.25,0.75) -- (3.25,1) -- (3,1) -- (3,0.75);
\draw[-] (3,1.75) -- (3.25,1.75) -- (3.25,2) -- (3,2) -- (3,1.75);
\draw[-] (3,2.75) -- (3.25,2.75) -- (3.25,3) -- (3,3) -- (3,2.75);
\draw[-] (3,3.75) -- (3.25,3.75) -- (3.25,4) -- (3,4) -- (3,3.75);
\draw[-] (2,1.75) -- (2.5,1.75) -- (2.5,2) -- (2,2) -- (2,1.75);
\draw[-] (2.25,1.75) -- (2.25,2);
\draw[-] (2,1.5) --(2.125,1.5) -- (2.125,1.625) -- (2,1.625) -- (2,1.5);
\draw[-] (2,3.75) -- (2.5,3.75) -- (2.5,4) -- (2,4) -- (2,3.75);
\draw[-] (2.25,3.75) -- (2.25,4);
\draw[-] (2,3.5) --(2.125,3.5) -- (2.125,3.625) -- (2,3.625) -- (2,3.5);
\draw[-] (0,3.75) -- (1,3.75) -- (1,4) -- (0,4) -- (0,3.75);
\draw[-] (0.25,3.75) -- (0.25,4);
\draw[-] (0.5,3.75) -- (0.5,4);
\draw[-] (0.75,3.75) -- (0.75,4);
\draw[-] (0,3.5) -- (0.125,3.5)-- (0.125,3.625) -- (0,3.625) -- (0,3.5);
\draw[-] (0.5,3.5) -- (0.625,3.5)-- (0.625,3.625) -- (0.5,3.625) -- (0.5,3.5);
\draw[-] (0,3) --(0.125,3) -- (0.125,3.25) -- (0,3.25) -- (0,3);
\draw[-] (0,3.125) -- (0.125,3.125);
\draw[-] (0.1875,3) -- (0.25,3) -- (0.25,3.0625) -- (0.1875,3.0625) --
         (0.1875,3);
\draw[-] (0,-0.1) -- (1,-0.1);
\draw[-] (2,-0.1) -- (2.5,-0.1);
\draw[-] (3,-0.1) -- (3.25,-0.1);
\draw[-] (3.5,-0.1) -- (3.625,-0.1);
\draw[] (3,1)     node[left]  {$1$};
\draw[] (3,3)     node[left]  {$2$};
\draw[] (0.5,3.5) node[below] {$3$};
\end{tikzpicture}
\caption{In Example~\ref{Ville}, the projection of the attractor onto the
  $x$-axis is a countable union of intervals.}
\label{Villefig}
\end{center}
\end{figure}
\end{example}

We remark that the different contraction ratios are not essential in
Example~\ref{Ville} and, with four maps, it is possible to construct a
homogeneous iterated function system having a projection which is a countable
union of intervals. Indeed, this is achieved by replacing the maps $f_1$ and
$f_2$ by three maps with contraction ratio $\frac 13$ in Example~\ref{Ville} and
using the same contraction $\frac 13$ also in the map $f_3$.

Our last example demonstrates the weak separation condition for projections.

\begin{example}\label{WSCexample}
Fix $n\in\N\setminus\{1\}$ and let $\cF$ be any nonempty subset of the family 
\[
\{(x,y)\mapsto\frac 1n(x,y)+\frac 1n(k,j)\mid k,j\in\N\cup\{0\}\text{ with }
0\le k, j \le n-1\}.
\]
Then $\cF$ generates a self-similar set $K$ on $[0,1]^2$
which is invariant under the multiplication by $n\mod 1$. Now the projection of
$K$ in any rational direction satisfies the weak separation condition, see
\cite[Remark 1]{Ru}. This is due to the fact that, for every rational $\theta$,
the orbit $\{k\theta\mod 1\mid k\in\N\}$ is finite. To be more precise about
the rational directions, observe that, for any vector $(a,b)\ne (0,0)$ with
$a,b\in\Z$, the map $(x,y)\mapsto ax+by$ defines a rational
projection up to scaling (which does not affect the dimension)
and every rational projection can be defined in this way. Thus, for this class
of "integral" self-similar sets, we can find a dense set of directions for
which the visible part problem is well understood due to Proposition~\ref{WSC}.
\end{example}

\begin{remark}\label{OSCneeded}
In the proof of Theorem~\ref{main}, the weak separation condition
guarantees that the
number $M$ in \eqref{localsubset} is independent of $r$. The finiteness of
$G(\cF)$, in turn, allows us to find a subsequence with a constant rotation part
in the argument after \eqref{visibleminiset}. Both these properties are
essential while we prove that the inclusion \eqref{visibleminiset} is valid also
for micro-sets. In Theorem~\ref{theorem2}, in addition to using
Theorem~\ref{main}, we apply the weak separation condition to conclude that the attractor $K$ is Ahlfors regular. In the proof of
Theorem~\ref{main}, we consider the closure of a visible part. Recall that the
restriction of the projection to a visible part is injective, but this is not
true for the closure. For example, let $F:=\{0\}\times [1,3]\cup\graph(f)$,
where $f\colon\mathopen]0,1\mathclose]\to\R$, $f(x)=\sin\frac 1x+2$, and
$\graph(f)$ is the graph of $f$. Then the visible part of $F$ from the $x$-axis
is $\{(0,1)\}\cup\graph(f)$, but its closure is $F$. So the closure of a visible
part may be much bigger than the visible part. However, Assouad dimension is an
upper bound for the box counting dimension, and the box counting dimension of
a set equals that of its closure. Thus, in this sense we do not loose anything
by considering the closure of a visible part. Nevertheless, we point out that
the upper bound in Theorem~\ref{main} is not always optimal. For example, in
Example~\ref{Meng}, the dimension of the penetrable part from the $x$-axis is
larger than one, but the visible part and its closure are 1-dimensional.
\end{remark}  

\bigskip

{\bf Acknowledgement} We thank Laurent Dufloux and Wen Wu for interesting
discussions. We also thank the referees for valuable comments,
especially pointing out the porosity argument for the proof of
Theorem~\ref{theorem2} and noting that only the weak separation condition is
needed.

\end{document}